\documentclass[a4paper,10pt]{article}
\usepackage{amsmath, amsthm, amssymb}
\usepackage{graphicx, subfigure}

\title{Piecewise-linear pseudodiagrams}
\author{Molly Durava\footnote{North Central College undergraduate}\\
        North Central College\\  
        madurava$@$noctrl.edu\\      
\\
        Neil R. Nicholson\\
	    North Central College\\
	    nrnicholson$@$noctrl.edu\\
\\
	    Jackson Ramsey\footnote{North Central College undergraduate Lederman Scholar}\\
	    North Central College\\
	    jsramsey$@$noctrl.edu
	}

\newtheoremstyle{dotless}{}{}{\itshape}{}{\bfseries}{}{ }{}
\def\Real{\hbox{I\kern-.1667em\hbox{R}}}
\theoremstyle{dotless}	
\input epsf

\newtheorem{theorem}{Theorem}[section]
\newtheorem{corollary}[theorem]{Corollary}
\newtheorem{lemma}[theorem]{Lemma}
\newtheorem{definition}[theorem]{Definition}

\begin{document}

\maketitle

\begin{abstract}
There are $2^n$ possible resolutions of a smooth pseudodiagram with $n$ precrossings.  If we consider piecewise-linear (PL) pseudodiagrams and resolutions that themselves are PL, certain resolutions of the pseudodiagram may not exist in $\mathbb{R}^3$.  We investigate this situation and its impact on the weighted resolution set of PL pseudodiagrams as well as introduce a concept specific to PL pseudodiagrams, the forcing number.  Our main result classifies the PL shadows whose weighted resolution sets differ from the weighted resolution set that would exist in the smooth case.
\end{abstract}

%\keywords{shadow, pseudoknot, weighted resolution set, piecewise linear}

%\ccode{Mathematics Subject Classification 2000: 57M25, 57M27}

\section{Introduction}
There have been various investigations into properties of smooth pseudoknots and their resolutions \cite{3,4,5,6,7}, but here we wish to focus our attention on those that are piecewise-linear (PL).

\begin{definition}
A \textit{pseudodiagram} is a knot diagram that may be missing some classical crossing information, with those crossings being called \textit{precrossings}.  If a pseudodiagram has no classical crossings, then it is called a \textit{shadow}.  An assignment of crossing information to every precrossing in a pseudodiagram is called a \textit{resolution} of the pseudodiagram.  See Fig. \ref{fig1}.
\end{definition}

\begin{figure}[ht!]
    \begin{center}

      \subfigure[]{%
         \label{fig:first}
        \includegraphics[width=0.2\textwidth]{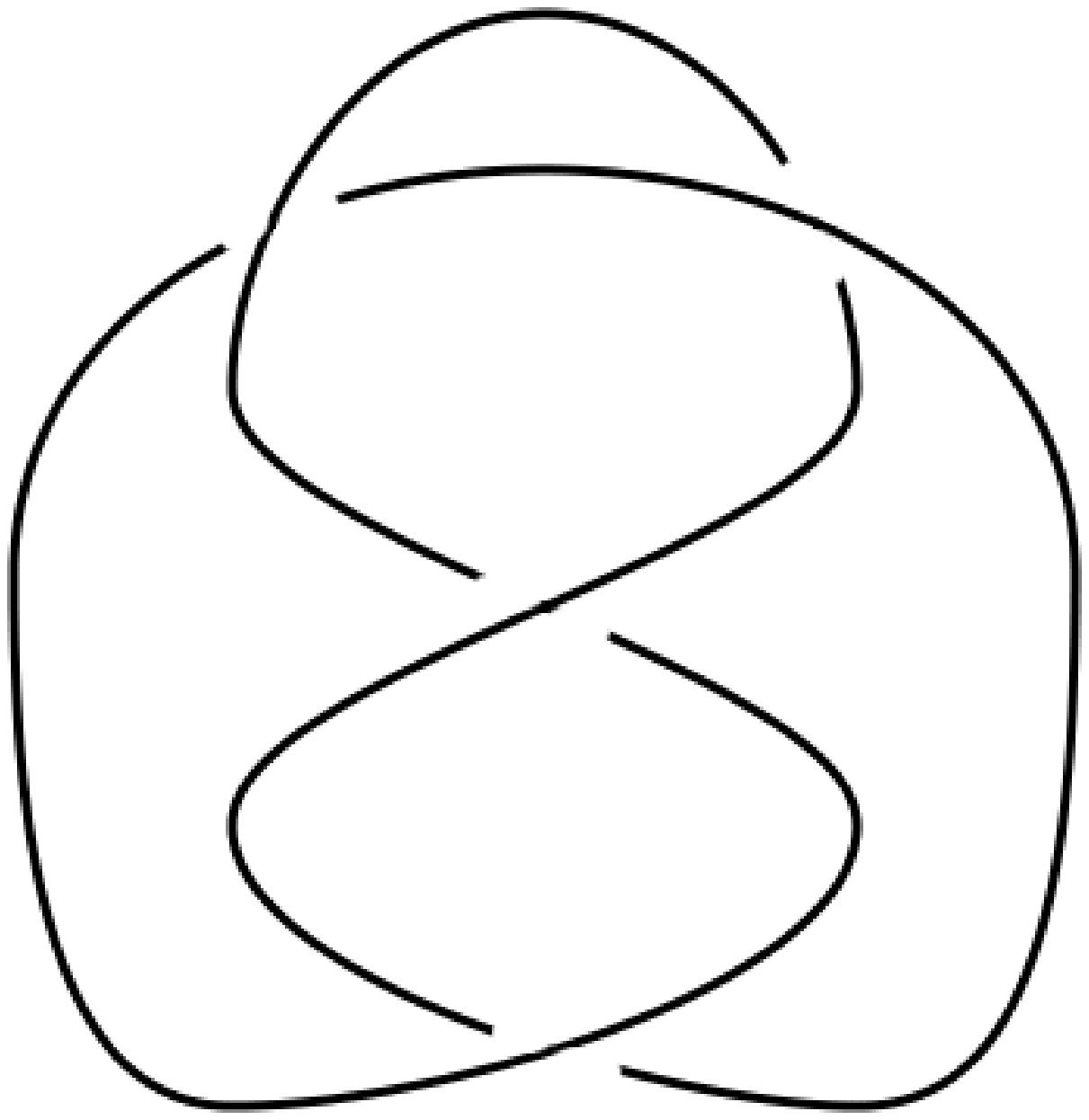}
        }%
        \subfigure[]{%
           \label{fig:second}
           \includegraphics[width=0.2\textwidth]{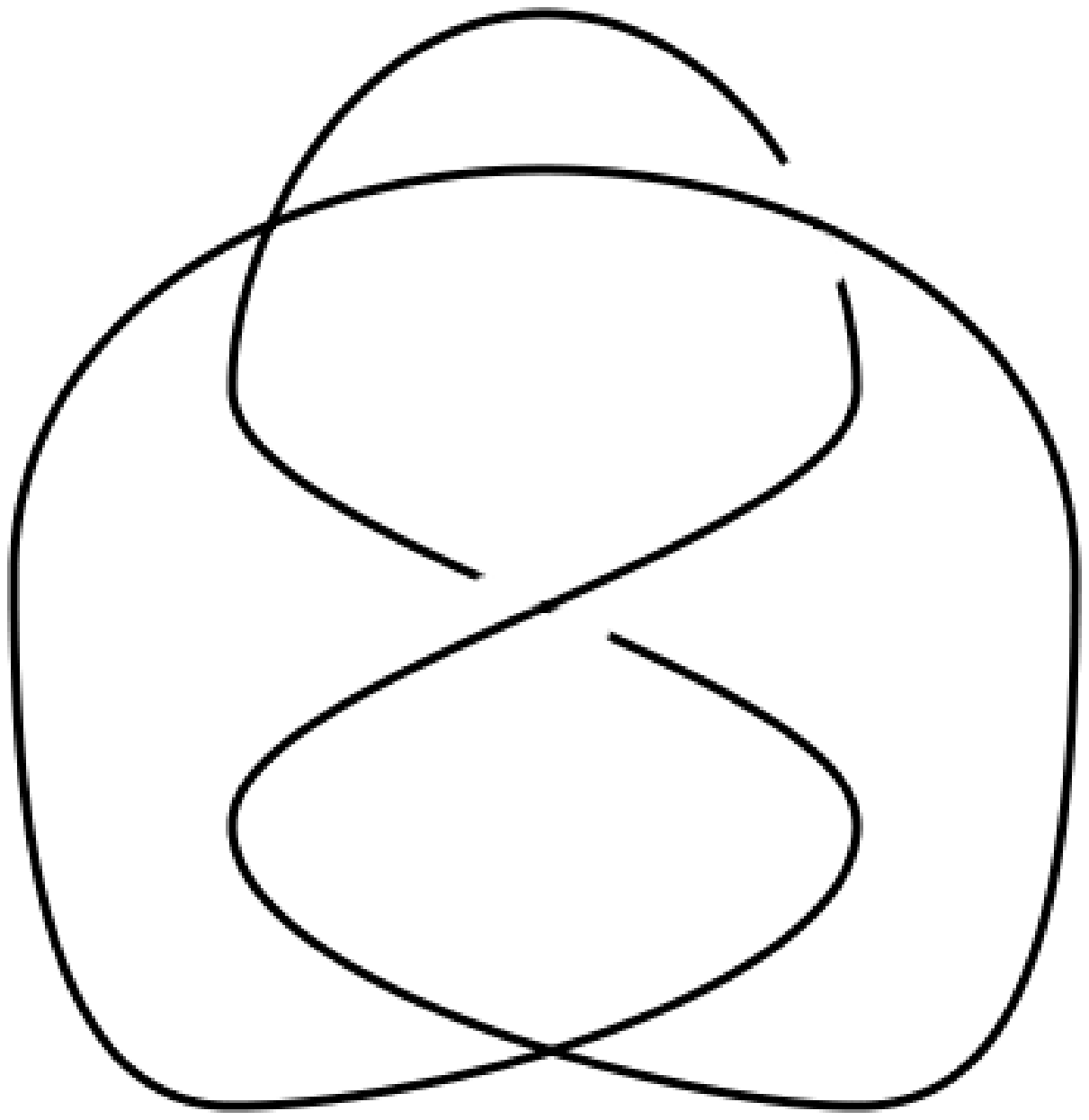}
        } %  ------- End of the first row ----------------------%
        \subfigure[]{%
            \label{fig:third}
            \includegraphics[width=0.2\textwidth]{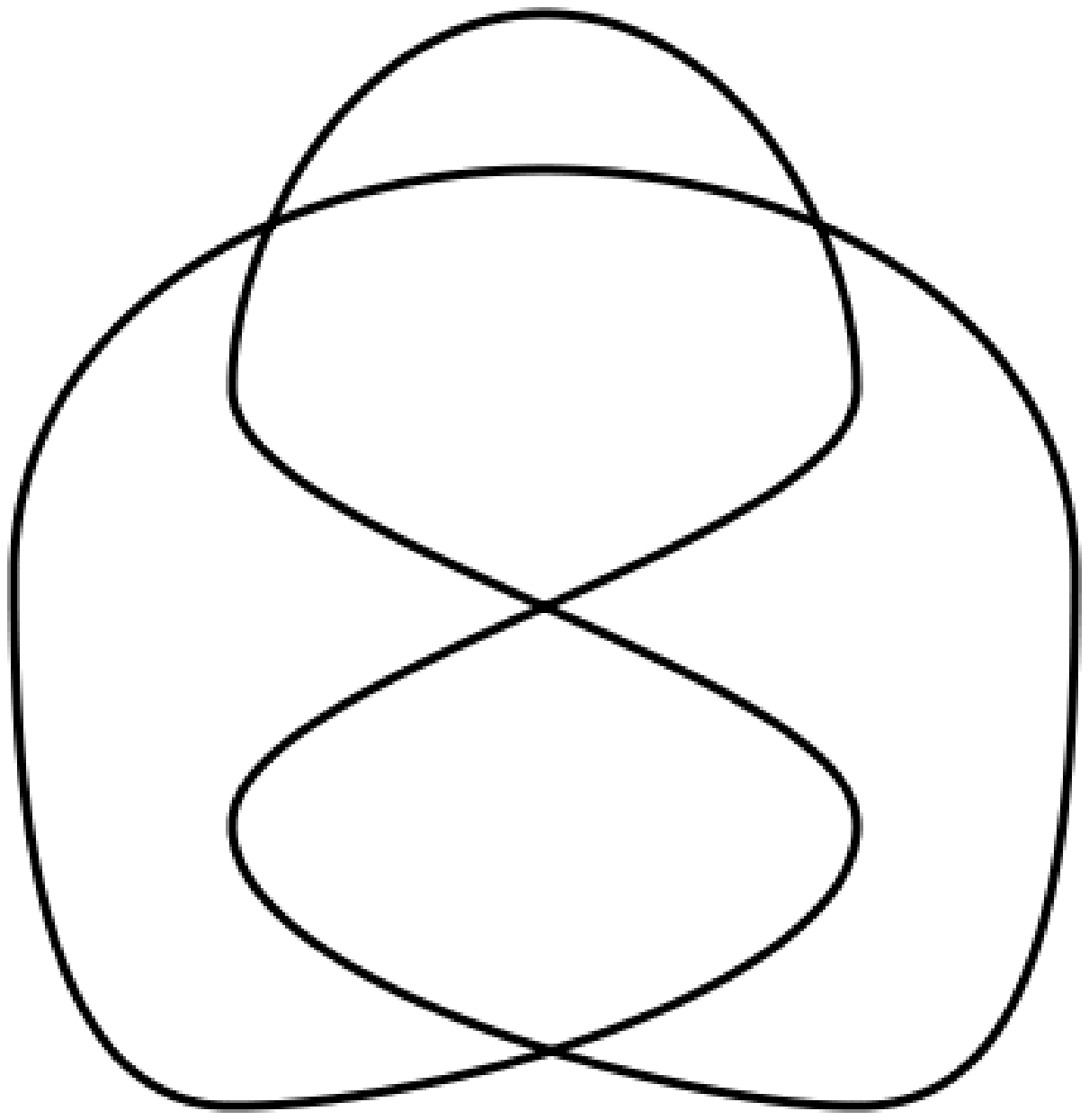}
        }%

    \end{center}
    \vspace*{8pt}
    \caption{Three pseudodiagrams.  Figure (a) is a resolution of both (b) and (c), with (c) being the shadow of the other two.}
   \label{fig1}
\end{figure}

In general, the resolutions of a piecewise-linear pseudodiagram need not themselves be PL diagrams.  However, for the purposes of this paper, we will require that they are.  This insistence is natural: a PL shadow is resolved to a PL knot.  

Smooth pseudodiagrams with $n$ precrossings have $2^n$ resolutions that exist in $\mathbb{R}^3$.  PL pseudodiagrams may not, however. 

\begin{definition}
A resolution of a PL pseudodiagram is called \textit{realizable} if it exists in $\mathbb{R}^3$ and \textit{nonrealizable} if it does not.
\end{definition}

Figure \ref{5star} is one example of a shadow and a nonrealizable resolution of it \cite{9}.  Theorem \ref{badshadowtheorem} classifies the other shadows that have nonrealizable resolutions.

\begin{figure}[ht!]
    \begin{center}

      \subfigure[]{%
         \label{fig:first}
        \includegraphics[width=0.2\textwidth]{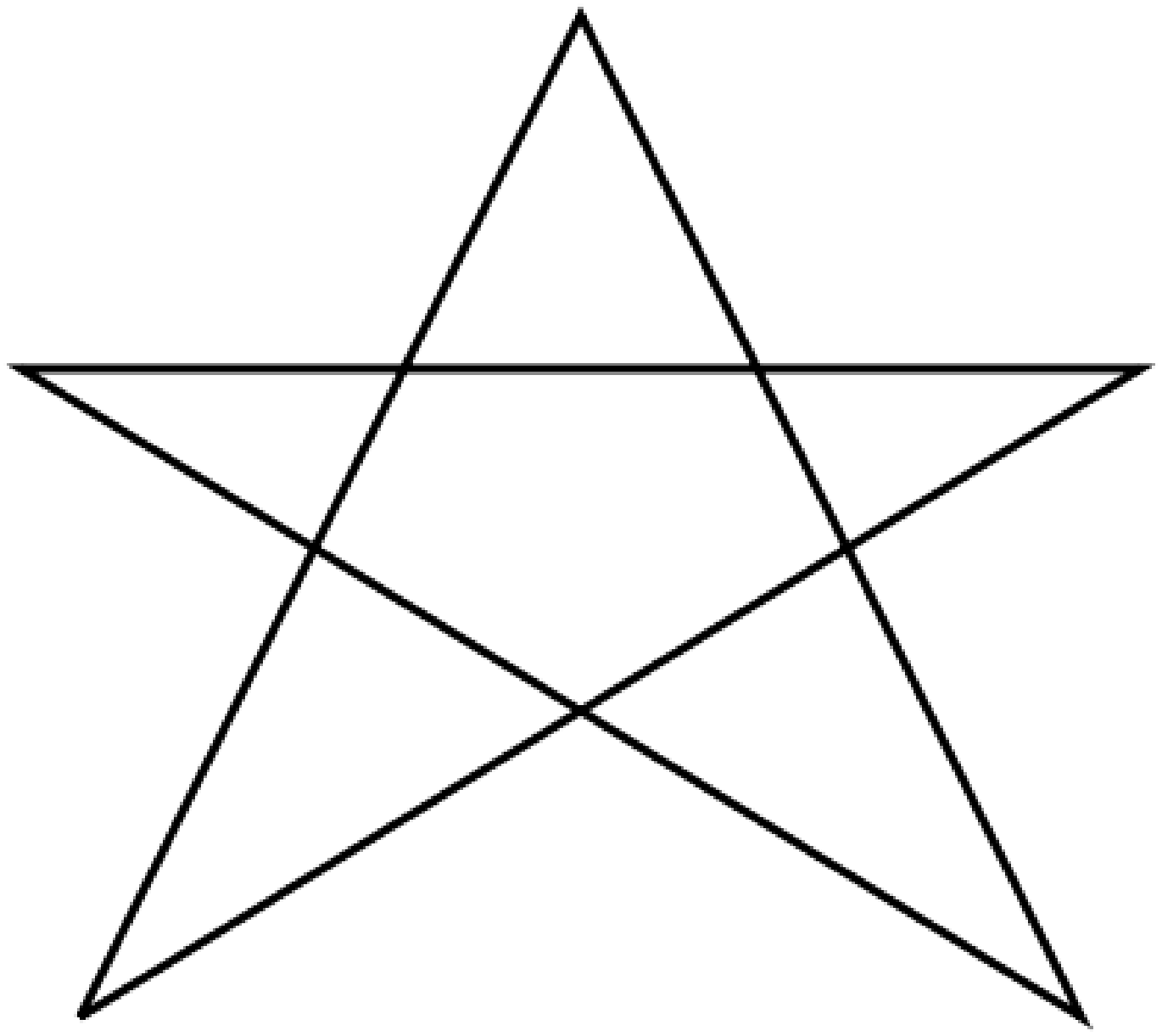}
        }
        \subfigure[]{%
           \label{fig:second}
           \includegraphics[width=0.2\textwidth]{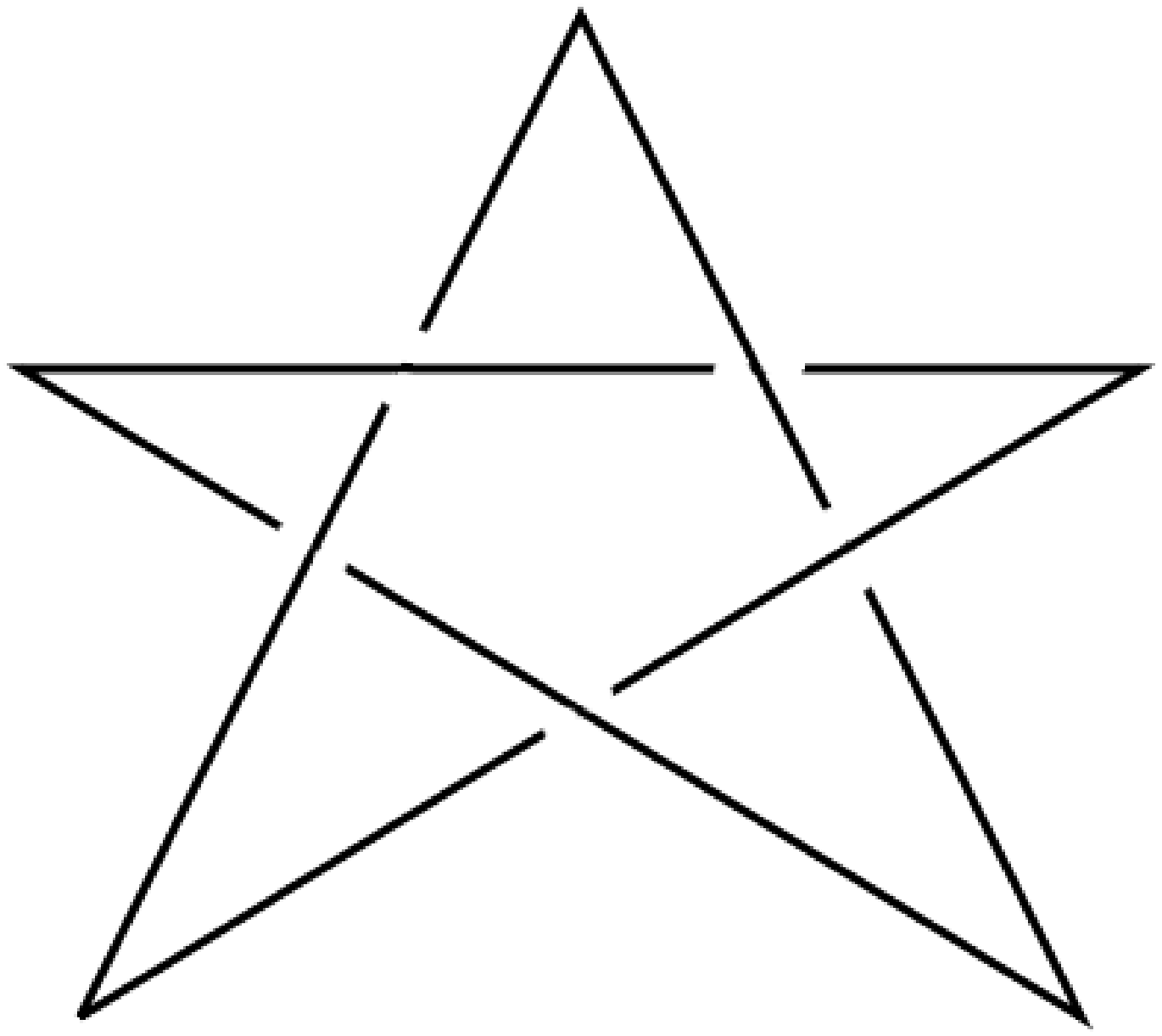}
        } 
    \end{center}
    \caption{A shadow and a nonrealizable resolution}
   \label{5star}
\end{figure}

The remainder of this paper is an investigation into weighted resolution sets and forcing numbers of PL shadows.  Weighted resolutions sets are an extension of the definition first appearing in \cite{4} but require exploration due to nonrealizable resolutions.  Next, we further explore the notion of realizability by introducing the forcing number for a diagram.  We conclude with possible directions for future work.

\section{Weighted Resolution Sets and Forcing Number} \label{WeReSets}

The notion of a weighted resolution set for a pseudodiagram was introduced by Henrich et al \cite{4}.  Because some resolutions of PL pseudodiagrams may be nonrealizable, we must adjust their definition.	

\begin{definition}
The \textit{weighted resolution set} (WeRe-set) of a PL pseudodiagram $D$ is the set of all ordered pairs $(K,p_K)$ and $(\emptyset, p_{\emptyset})$, where $K$ is a realizable resolution of $D$ and $p_K$ is the probability that $K$ is obtained from $D$ by randomly assigning crossing information to every precrossing in $D$ (with either assignment of crossing information to a precrossing being equally likely) and $p_{\emptyset}$ is the probability that the resolution is not realizable.
\end{definition}

A quick sketch of the 32 resolutions of Fig. \ref{5star}(a) shows that the shadow has WeRe-set $\{(0_1,\frac{20}{32}),(\emptyset,\frac{12}{32})\}$.  In the smooth case, the WeRe-set would be $\{(0_1,\frac{20}{32}),(3_1, \frac{10}{32}),(5_1, \frac{2}{32}) \}$ \cite{4}.  Besides the difference of PL resolutions being nonrealizable, note that in the smooth case the knot $5_1$ occurs as a resolution.  Because a nontrivial PL knot requires at least six edges \cite{8,10}, we know that the shadow of Fig. \ref{5star} cannot be resolved to a PL diagram of $5_1$.

For reference, we calculate here the WeRe-set for each of the shadows appearing in Fig. \ref{wereshadows}.  Figure \ref{wereshadows}(a) has WeRe-set $\{(0_1,\frac{6}{8}),(\emptyset,\frac{2}{8})\}$, while Fig. \ref{wereshadows}(b) has WeRe-set $\{(0_1,\frac{20}{32}),(3_1,\frac{2}{32}),(\emptyset,\frac{10}{32}) \}$ and Fig. \ref{wereshadows}(c) has WeRe-set $\{(0_1,\frac{20}{32}),(3_1,\frac{10}{32}),(\emptyset,\frac{2}{32}) \}$.  We note that all of these particular shadows have nonrealizable resolutions and this leads to a natural question: can we classify the shadows with such a property?  That is, which shadows, when considered in the PL sense, have a WeRe-set that differs if we were to consider the shadow in the smooth sense?  Lemma \ref{badshadowlemma} provides four such cases.

\begin{figure}[ht!]
    \begin{center}

      \subfigure[]{%
         \label{fig:first}
        \includegraphics[width=0.2\textwidth]{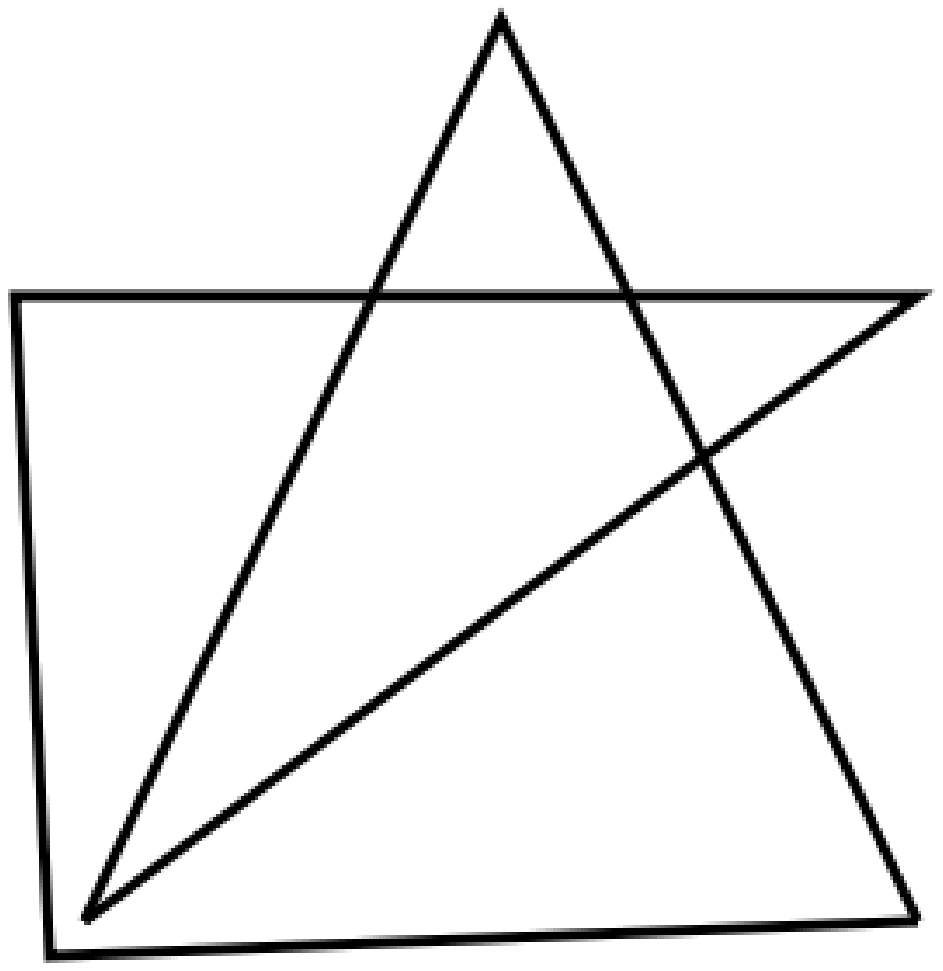}
        }
        \subfigure[]{%
           \label{fig:second}
           \includegraphics[width=0.2\textwidth]{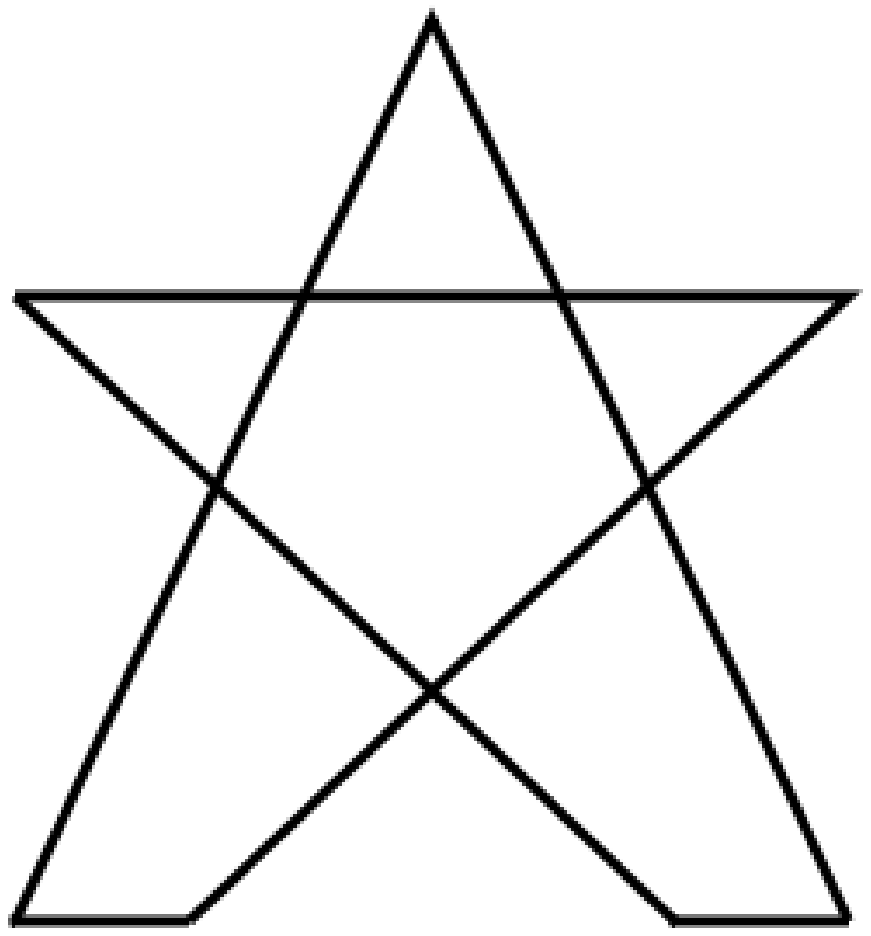}
        } 
        \subfigure[]{%
            \label{fig:third}
            \includegraphics[width=0.2\textwidth]{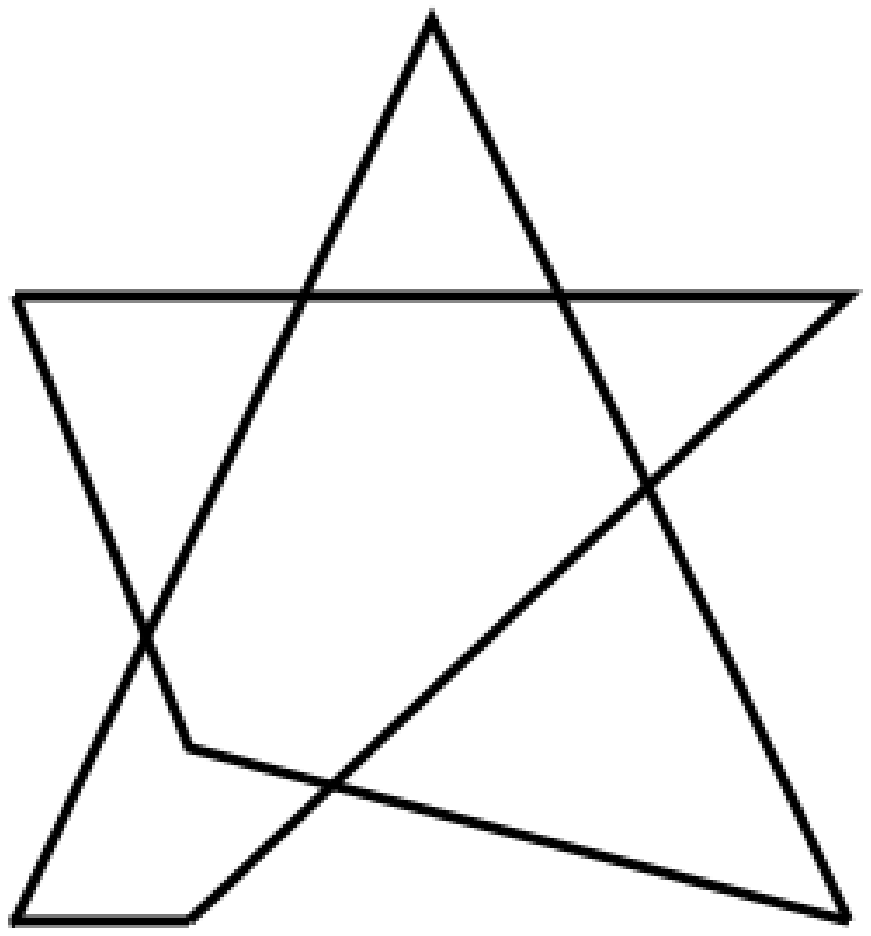}
        }
    \end{center}
    \caption{Shadows with calculated WeRe-sets}
   \label{wereshadows}
\end{figure}

\begin{lemma} \label{badshadowlemma}
If $S$ is a PL shadow that has a portion of it isotopic to one of the diagrams in Fig. \ref{badshadows} (or a mirror image of such a diagram), then $S$ has nonrealizable resolutions, and hence, the WeRe-set for $S$ differs from the WeRe-set when $S$ is considered to be a shadow of a smooth knot.
\end{lemma}
\begin{proof}
Observe that the resolutions of the PL shadows of Fig. \ref{badshadows} that appear in Fig. \ref{badshadowresolutions}, respectively, are not possible in $\mathbb{R}^3$.\\
\end{proof}

\begin{figure}[ht!]
    \begin{center}

      \subfigure[]{%
         \label{fig:first}
        \includegraphics[width=0.2\textwidth]{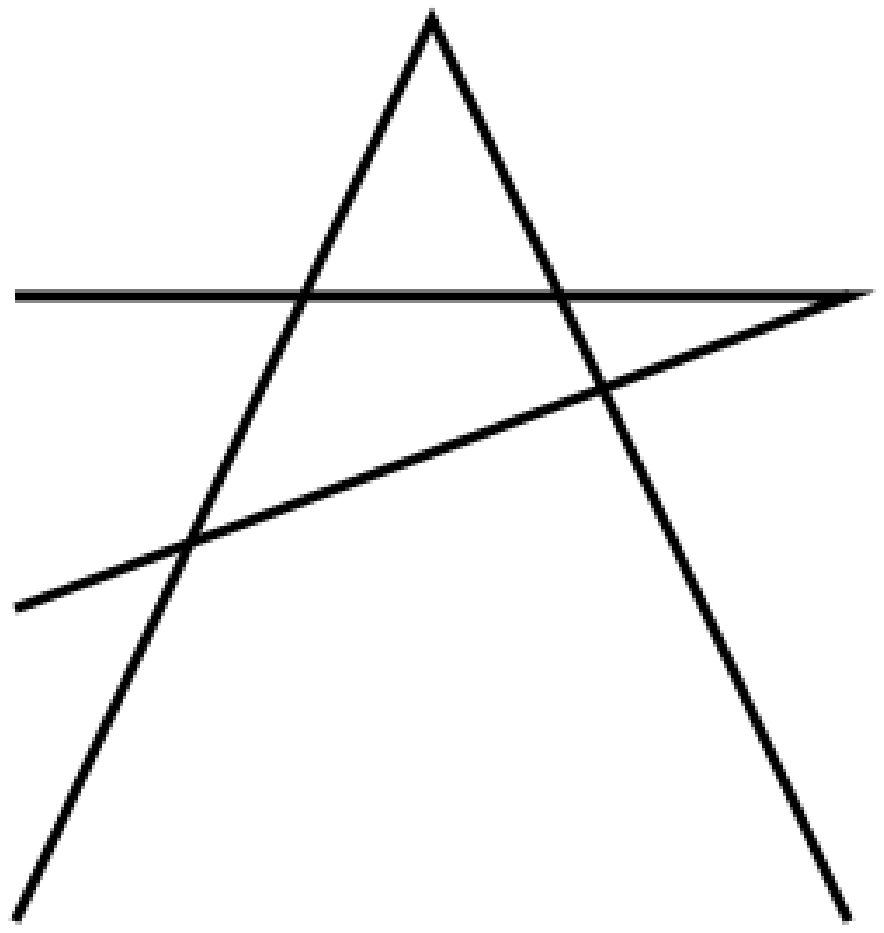}
        }
        \subfigure[]{%
           \label{fig:second}
           \includegraphics[width=0.2\textwidth]{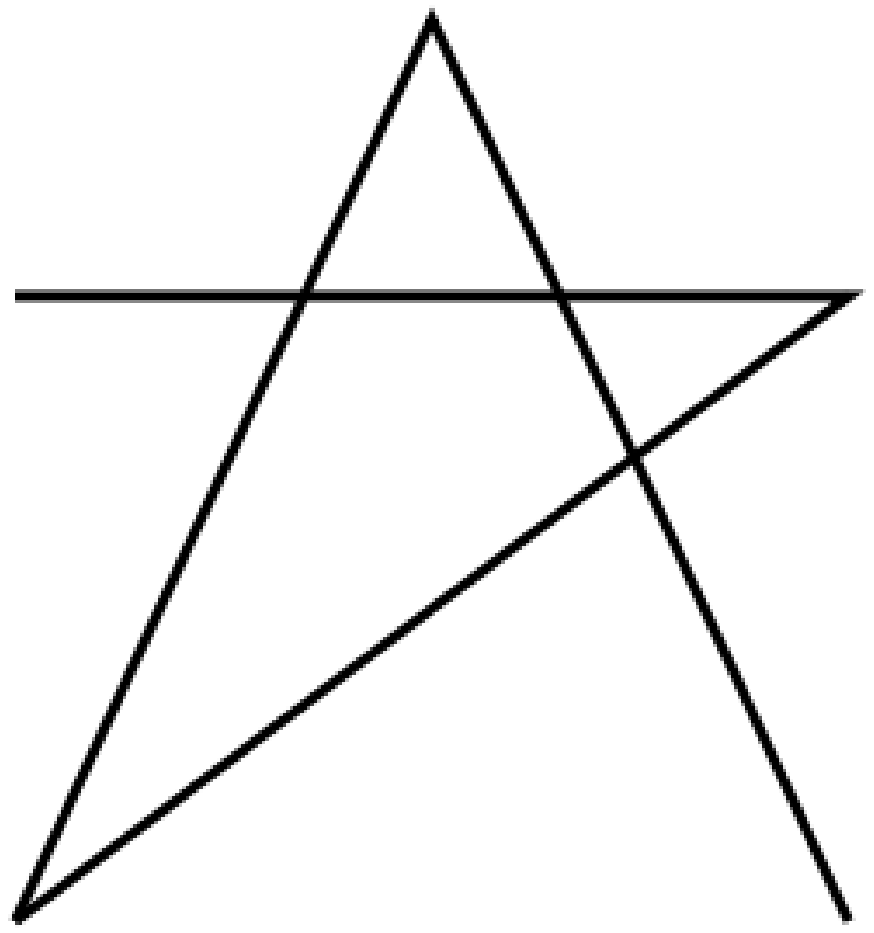}
        } 
        \subfigure[]{%
            \label{fig:third}
            \includegraphics[width=0.2\textwidth]{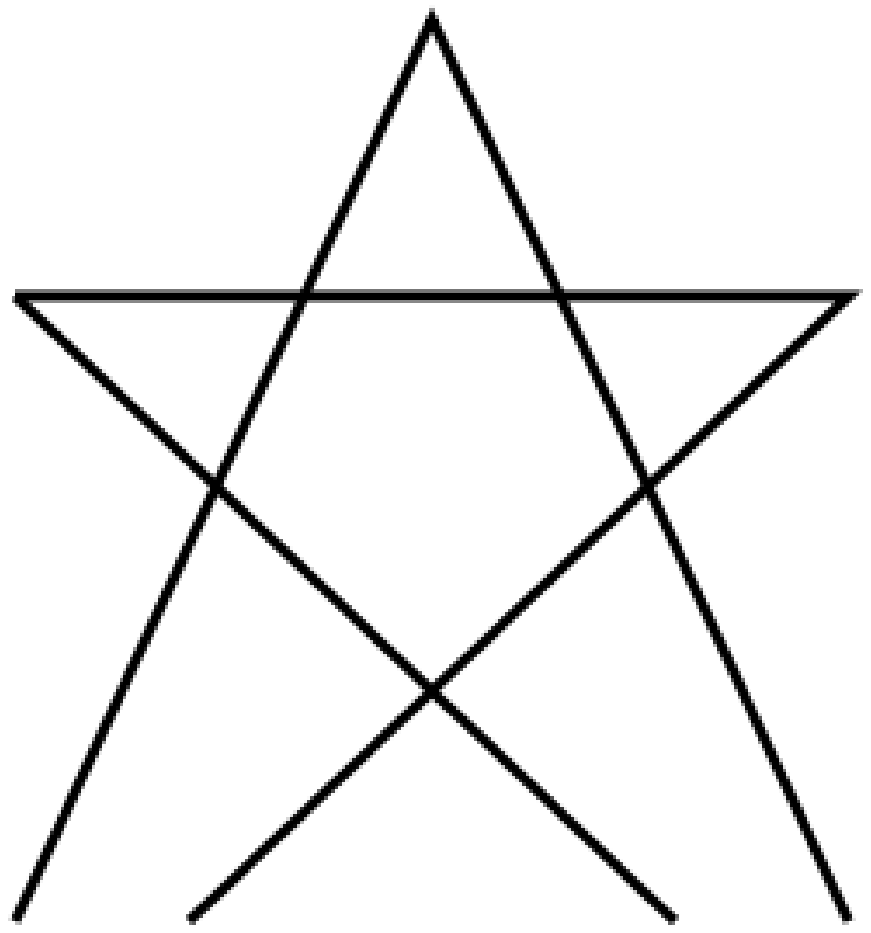}
        }
        \subfigure[]{%
            \label{fig:fourth}
            \includegraphics[width=0.2\textwidth]{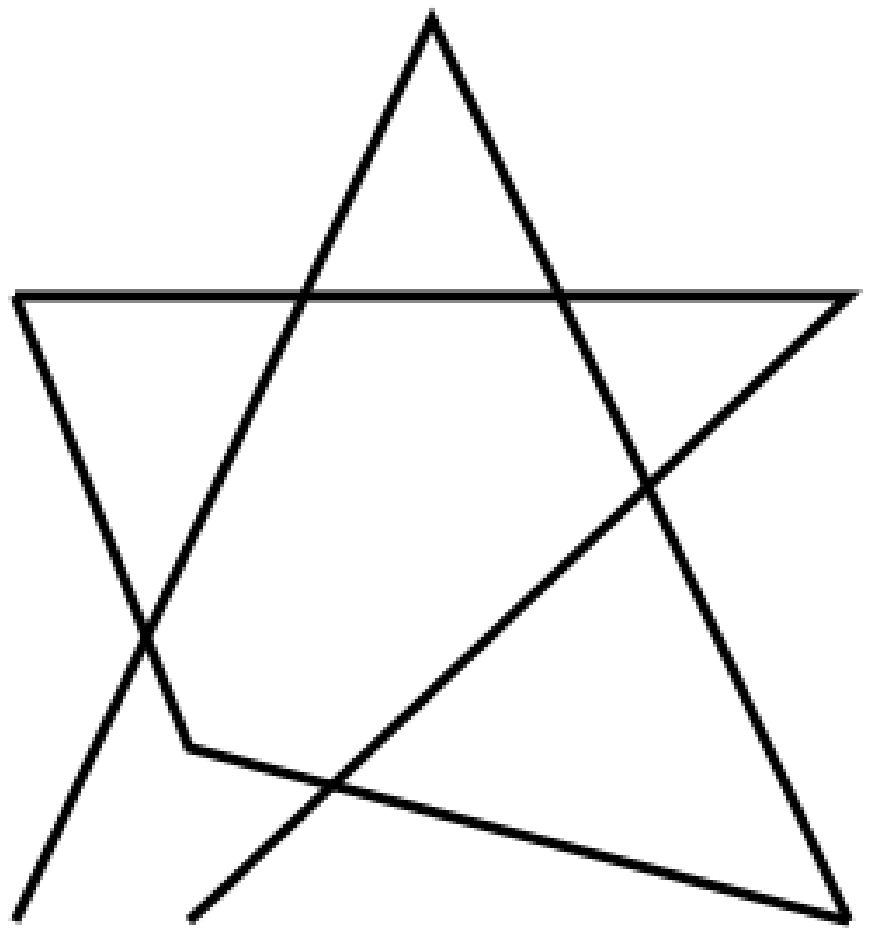}
        }
    \end{center}
    \caption{Portions of shadows with nonrealizable resolutions}
   \label{badshadows}
\end{figure}

\begin{figure}[ht!]
    \begin{center}

      \subfigure[]{%
         \label{fig:first}
        \includegraphics[width=0.2\textwidth]{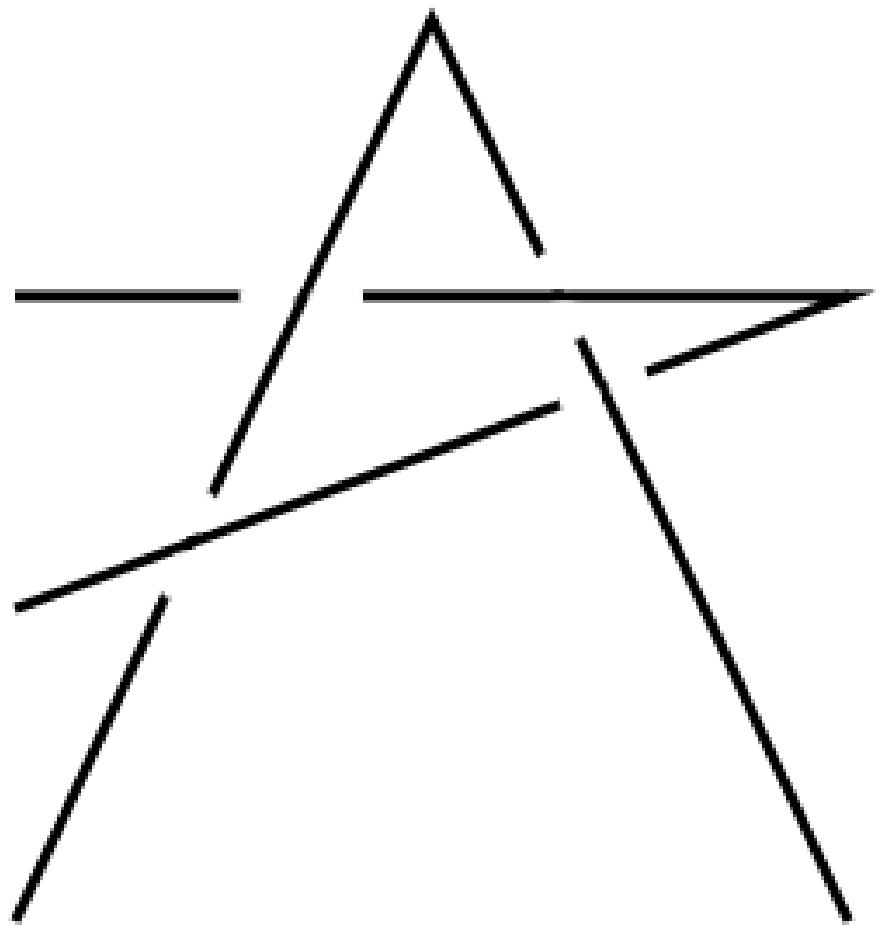}
        }
        \subfigure[]{%
           \label{fig:second}
           \includegraphics[width=0.2\textwidth]{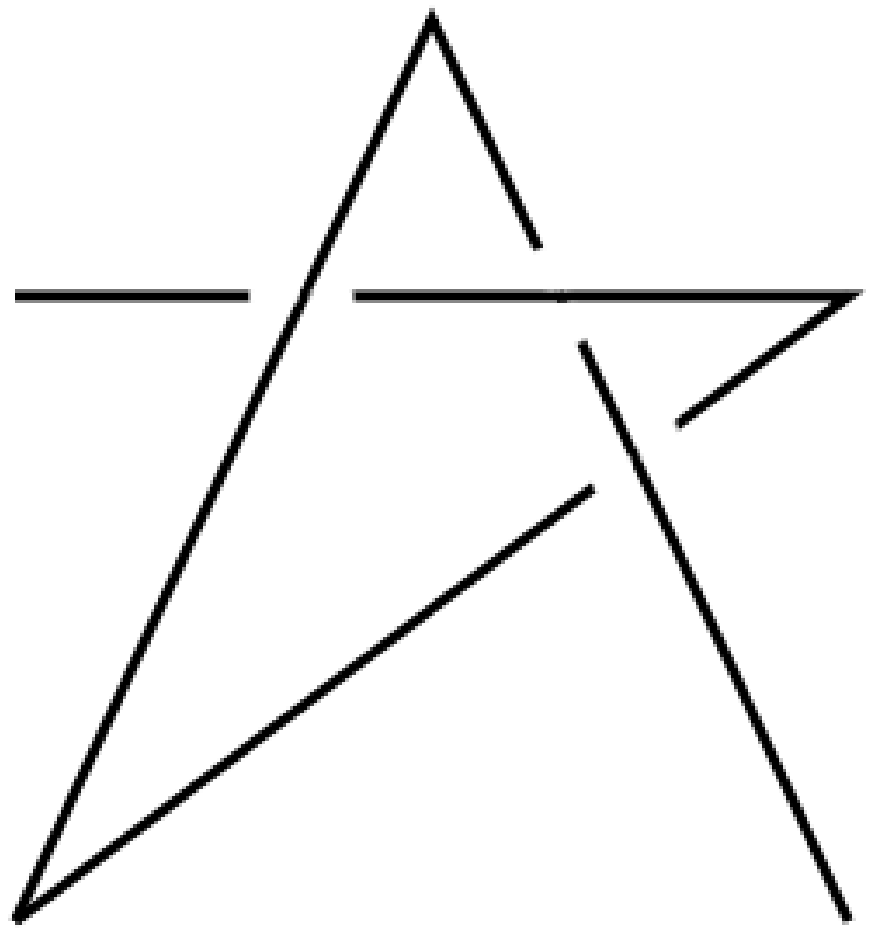}
        } 
        \subfigure[]{%
            \label{fig:third}
            \includegraphics[width=0.2\textwidth]{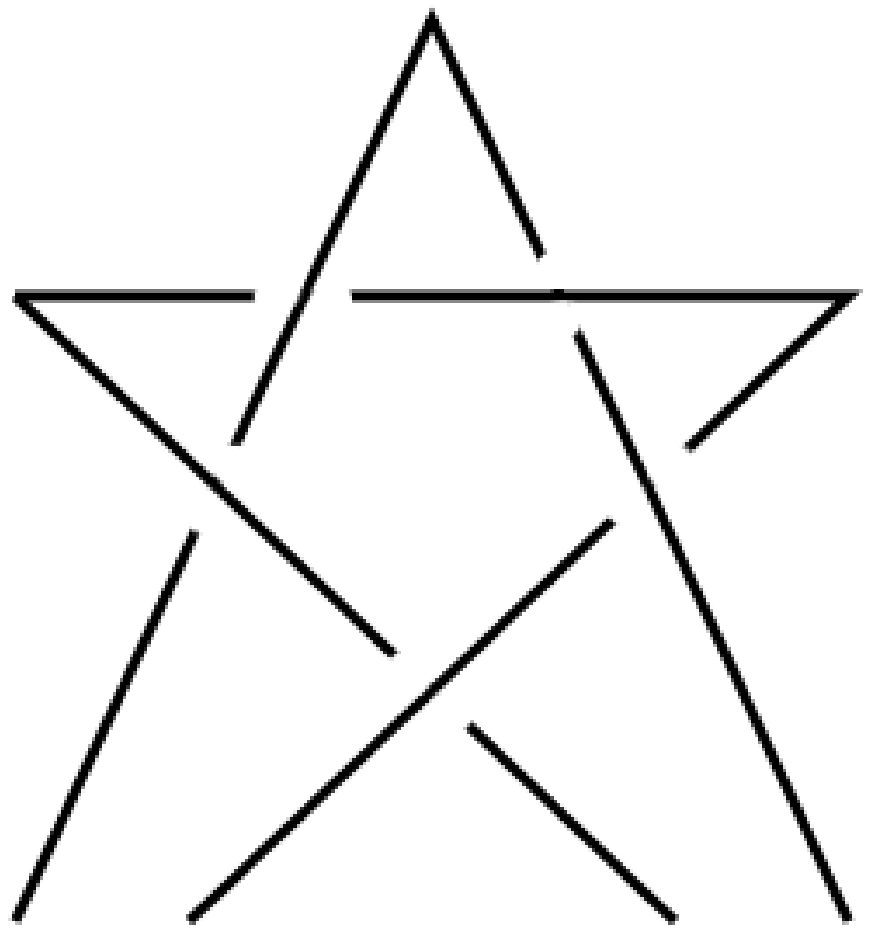}
        }
        \subfigure[]{%
            \label{fig:fourth}
            \includegraphics[width=0.2\textwidth]{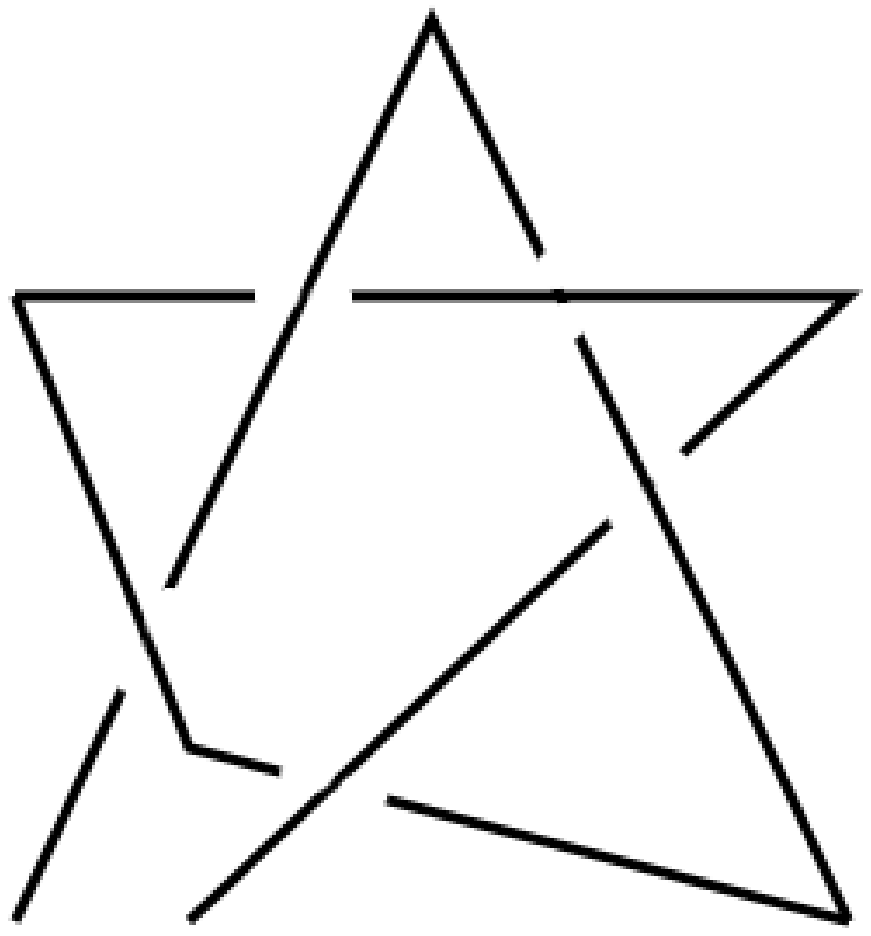}
        }
    \end{center}
    \caption{Nonrealizable resolutions of Fig. \ref{badshadows}}
   \label{badshadowresolutions}
\end{figure}

Are there other shadows with nonrealizable resolutions?  Note that a resolution $R$ of a shadow $S$ is nonrealizable if an edge of $R$ is forced to ``bend;" that is, there exists a plane that the edge crosses yet the edge has both of its endpoints on the same side of the plane.  The following theorem, our main result, proves that Lem. \ref{badshadowlemma} is a complete categorization of such shadows.  We will be using the following notation.  A PL shadow $S$ consists of $n$ distinct points \\ $v_1$, $v_2$, ..., $v_n$, where $v_i = (x_i,y_i,0)$, in the plane and $n$ linear segments, called edges, $e_i = v_{i-1}v_{i}$ (considering the $v_i$ cyclically), so that any two edges that intersect must do so transversally.  A resolution of $S$ is an assignment to each $v_i$ a point in $\mathbb{R}^3$, $\overline{v_i} = (x_i,y_i,z_i)$, so that no two resolved edges $\overline{e_i} = \overline{v_{i-1}} \hspace{2pt} \overline{v_{i}}$ intersect except at their endpoints.  

\begin{theorem} \label{badshadowtheorem}
Any shadow that has nonrealizable resolutions must have a portion of it isotopic to one of the figures in Fig. \ref{badshadows}. 
\end{theorem}
\begin{proof}
If $S$ is a shadow with resolution $R$, using the above notation, then there are two types of planes to consider: those formed by two adjacent edges of $R$ and those not containing two adjacent edges of $R$.  If $e_i$ and $e_{i+1}$ are two adjacent edges of $S$, then $\overline{e_i}$ and $\overline{e_{i+1}}$ will always create a plane in $\mathbb{R}^3$, no matter if the vertices of $R$ are translated or not.  Thus, we must determine what resolutions are nonrealizable with regards to these planes.  That is, starting with two adjacent edges of $S$, what other arrangements of edges of $S$ could lead to potentially nonrealizable resolutions?  This has been done \cite{9}, yielding the four cases of Fig. \ref{badshadows}.

Now suppose a plane $P$ does not contain adjacent edges of $R$.  Let us assume $P$ does contain an edge $\overline{e_i} = \overline{v_{i-1}} \overline{v_i}$ of $R$ and a third point $p$, not on $\overline{e_{i-1}}$ or $\overline{e_{i+1}}$ of $R$ (lest $P$ contains two adjacent edges of $R$).  $P$ can be projected to the $xy$-plane.  If $e_j$ is an edge of $S$ intersecting this region, then we can guarantee $\overline{v_{j-1}}$ and $\overline{v_j}$ lie on opposite sides of $P$, since not both $\overline{v_{j-1}}$ and $\overline{v_j}$ are endpoints of $\overline{e_i}$ or the edge $p$ lies on. The points $\overline{v_{j-1}}$ and $\overline{v_j}$ can be isotoped ($(x_j,y_j,z_j)$ of $R$ can be translated to $(x_j,y_j,z_j + \epsilon_j)$, for example, where $\epsilon_j \in \mathbb{R}$) while still preserving the knot type of $R$. See Fig. \ref{plane}.  

The above argument holds for any plane not containing two adjacent edges of $R$, and thus, the result holds.
\end{proof}

\begin{figure}[ht!]
	    \begin{center}
        \includegraphics[width=0.3\textwidth]{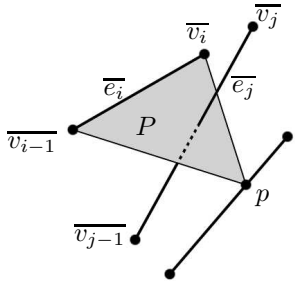}
        \put(-123,50){$\overline{v_{i-1}}$}
        \put(-87,70){$\overline{e_{i}}$}
        \put(-56,91){$\overline{v_{i}}$}
        \put(-98,15){$\overline{v_{j-1}}$}
        \put(-30,98){$\overline{v_{j}}$}
        \put(-39,71){$\overline{e_{j}}$}
        \put(-30,30){$p$}
        \put(-75,55){$P$}
        \end{center}
    \caption{Disregarding planes not formed by adjacent edges} 
   \label{plane}
\end{figure}

We immediately see that all resolutions of one category of shadows are realizable.

\begin{corollary}
The shadow of the PL $(n,2)$-torus knot, for $n$ odd, $n \geq 7$ (as pictured in Fig. \ref{torus}) has $2^n$ realizable resolutions.
\end{corollary}
\begin{proof}
For such $n$, these shadows contain no portion of their diagrams isotopic to those in Fig. \ref{badshadows}.
\end{proof}

\begin{figure}[ht!]
	\begin{center}
    \subfigure[]{%
         \label{fig:first}
        \includegraphics[width=0.2\textwidth]{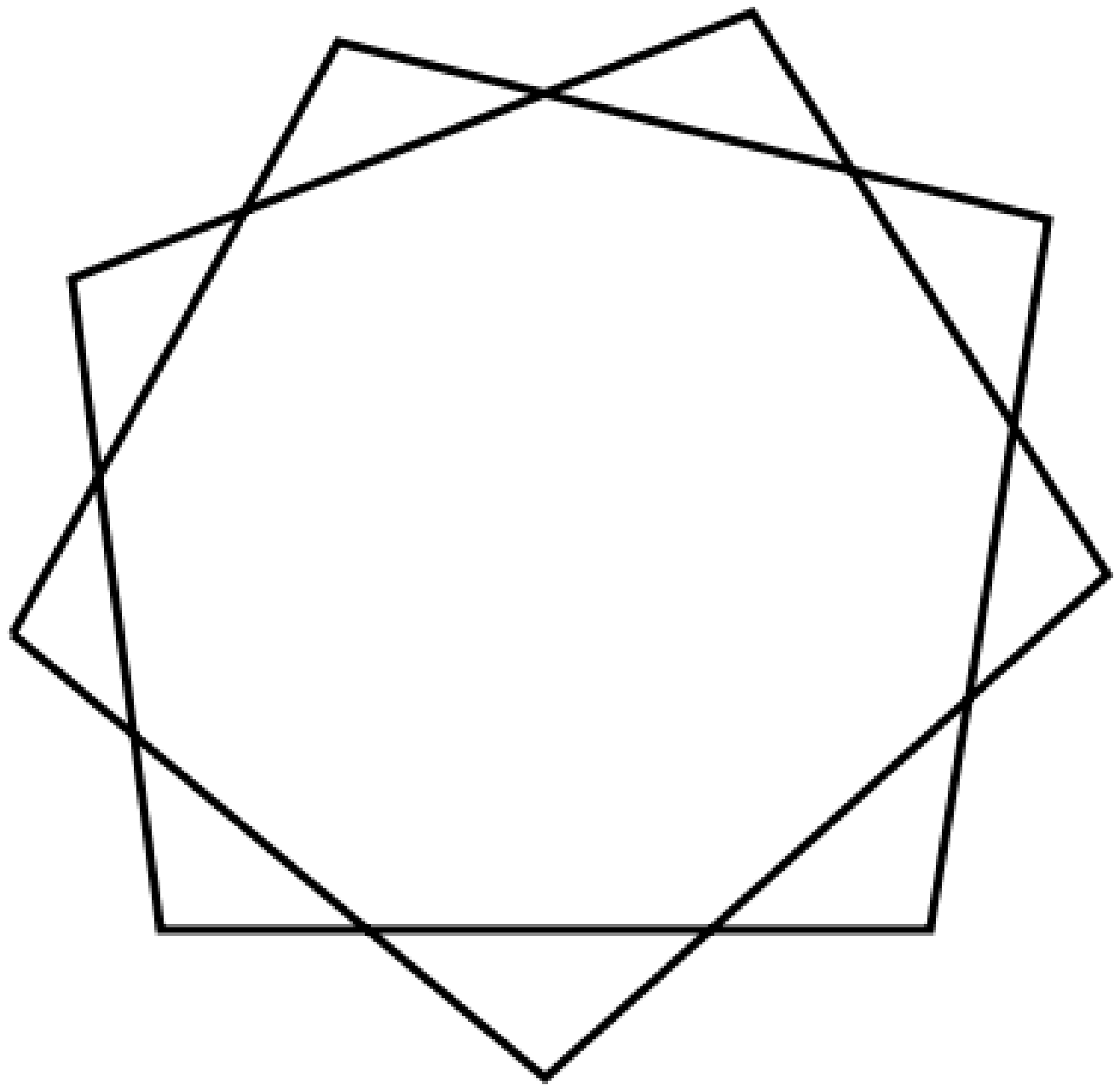}
        }
        \subfigure[]{%
           \label{fig:second}
           \includegraphics[width=0.2\textwidth]{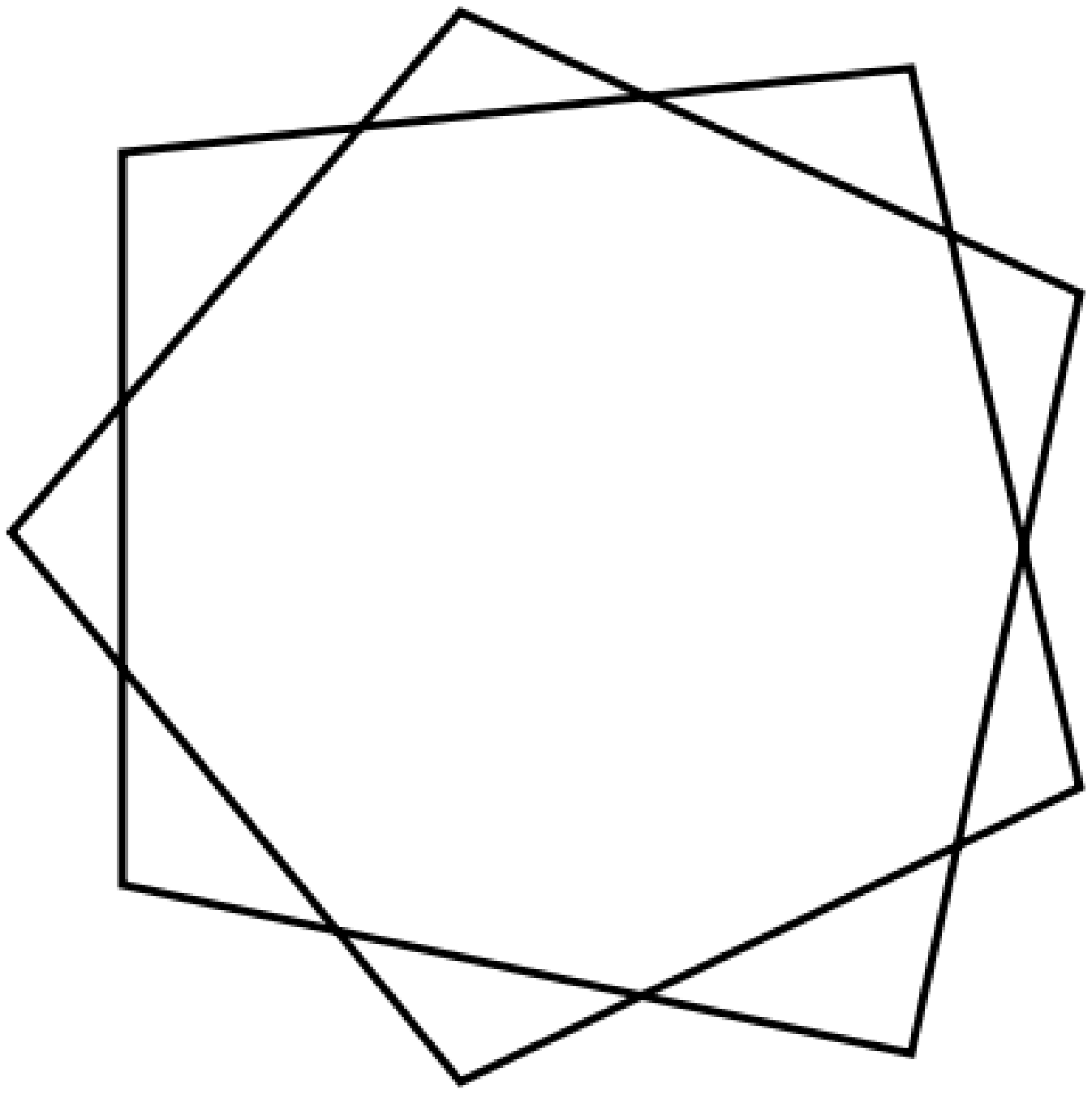}
        } 
    \end{center}
    \caption{Shadows of the PL $(7,2)$- and $(9,2)$-torus knots, respectively} 
   \label{torus}
\end{figure}

If one starts with the shadow of Fig. \ref{5star}(a) and begins choosing resolutions for the precrossings, with the goal of creating a realizable resolution, then there may come a point when there is not a choice of resolution for a particular crossing.  It may happen that, in order to realize the resolution, the precrossing is forced to be assigned one particular type of crossing.  This idea introduces the following two definitions.

\begin{definition} 
Let $D$ be a PL pseudodiagram with $P$ the set of precrossings of $D$.  Then, $S \subset P$ is said to \textit{force} $D$ if there exists an assignment of crossing information to $s \in S$ so that all crossings of $P - S$ must be resolved one particular way in order to realize the resolution of $D$ in $\mathbb{R}^3$.
\end{definition}

\begin{definition} 
If $D$ is a PL pseudodiagram, then the \textit{forcing number of D}, $f(D)$, is the size of the smallest set of precrossings of $D$ that forces $D$.
\end{definition}

\begin{lemma} \label{starlemma}
If $D$ is the shadow of Fig. \ref{5star}, then $f(D) = 2$.
\end{lemma}
\begin{proof}
It is clear that no pseudodiagram has forcing number $1$, so it suffices to find a set of two precrossings of $D$ that forces $D$.  Choose the resolutions of the two crossings as pictured in Fig. \ref{5starforce}(a).  By TLem. \ref{badshadowlemma}, two precrossings are forced to be resolved a certain way, as in Fig. \ref{5starforce}(b), for if either resolution were switched, regardless of how the remaining crossings are resolved, the resolution of the shadow would be nonrealizable (see Fig. \ref{badshadowresolutions}(b)).  Once resolved, by a similar argument, the final precrossing is forced to be resolved as in Fig. \ref{5starforce}(c).
\end{proof}

\begin{figure}[ht!]
    \begin{center}

      \subfigure[]{%
         \label{fig:first}
        \includegraphics[width=0.2\textwidth]{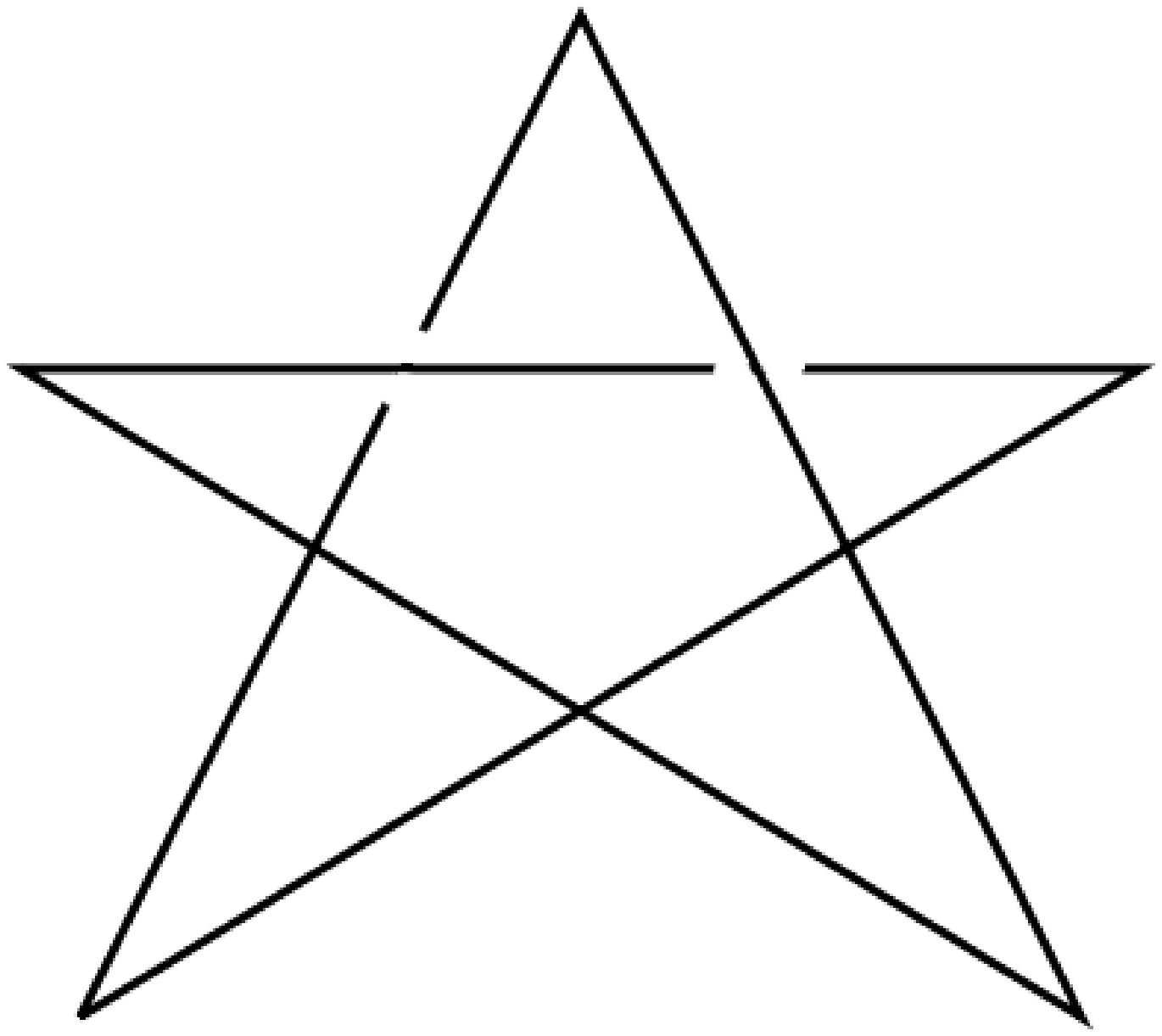}
        }
        \subfigure[]{%
           \label{fig:second}
           \includegraphics[width=0.2\textwidth]{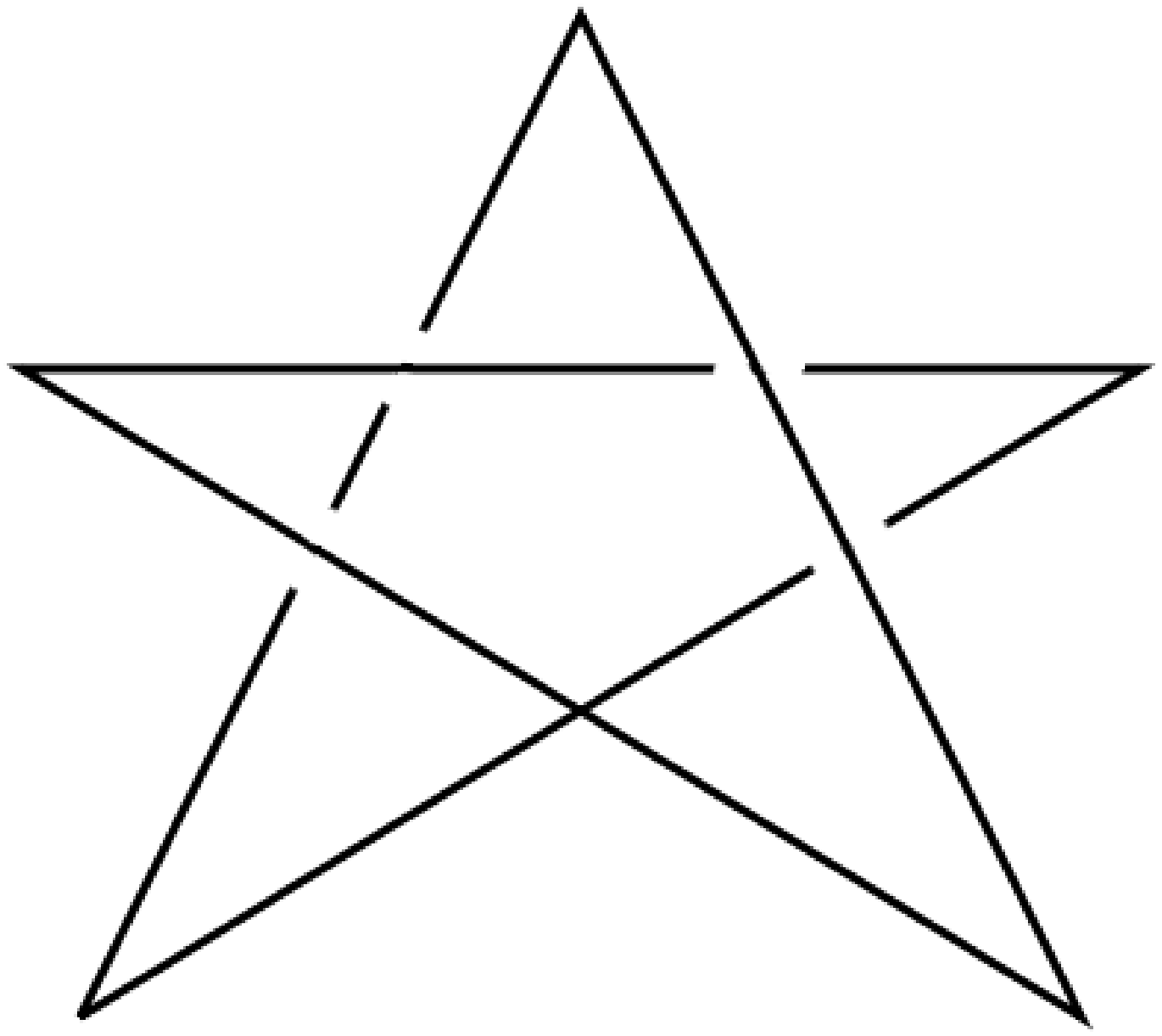}
        } 
        \subfigure[]{%
            \label{fig:third}
            \includegraphics[width=0.2\textwidth]{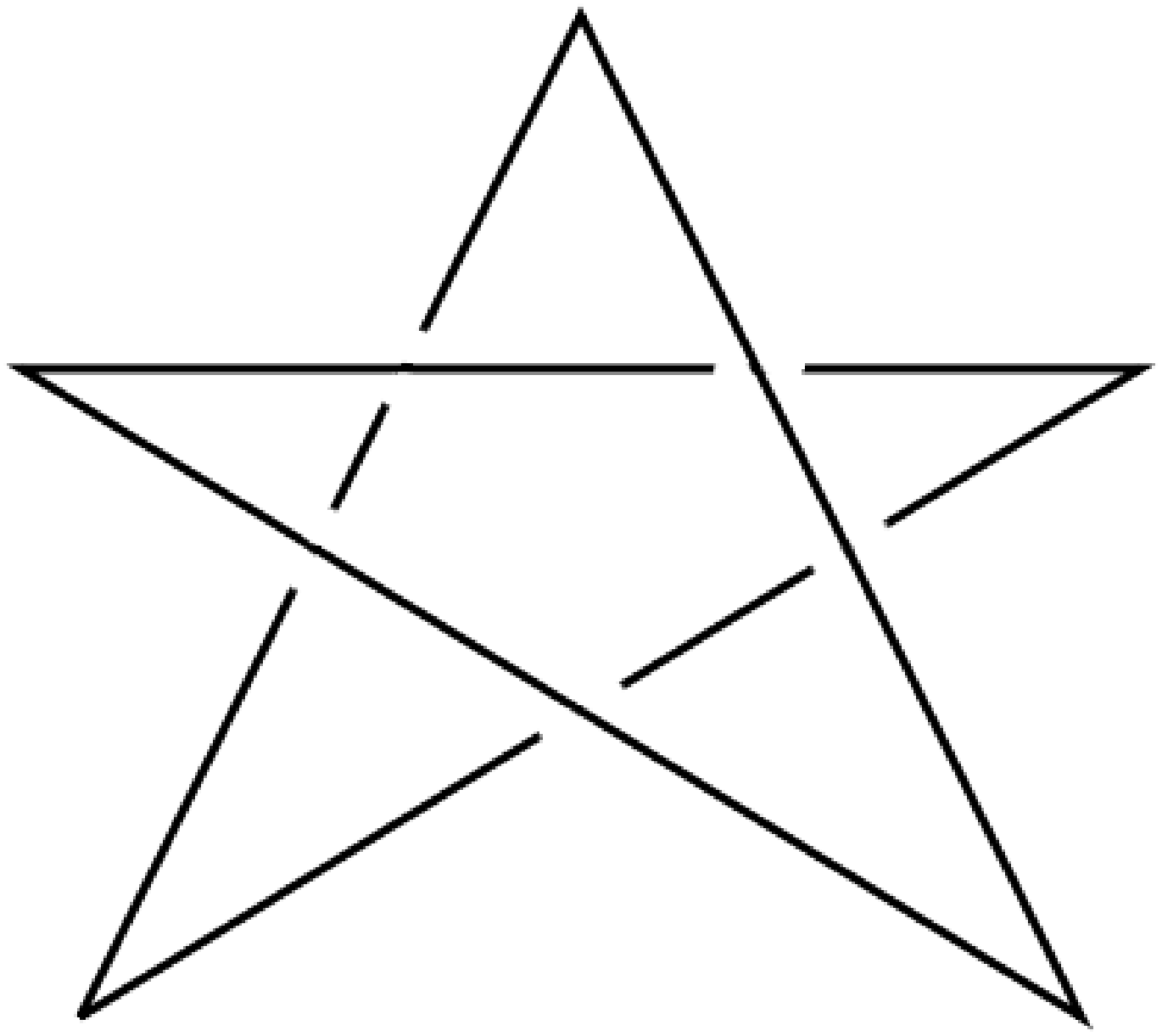}
        }
    \end{center}
    \caption{Forcing Fig. \ref{5star}(a)}
   \label{5starforce}
\end{figure}

What if a shadow contains multiple portions isotopic to those in Fig. \ref{badshadow}?

\begin{corollary}
Let $S$ be a shadow with $n$ portions of it isotopic to those appearing in Fig. \ref{badshadows}.  If $m$ is the maximum number of crossings of $S$ that can be forced, then

\begin{equation}
m \leq n.
\end{equation}
\end{corollary}
\begin{proof}
Note that a precrossing in a pseudodiagram $S$ can be forced only if the other choice of resolution for it results in a nonrealizable resolution.  The only situation in which this could occur is if a portion of $S$ is isotopic to one of  Fig. \ref{badshadowresolutions} (or other nonrealizable resolutions of Fig. \ref{badshadows}) with one of the classical crossings yet still a precrossing.  Then, there is only one possibility for resolving this precrossing, to yield a realizable resolution.  This is true for each region of $S$ isotopic to one of Fig. \ref{badshadows}, proving the result.
\end{proof}

\section{Future Questions} \label{future}
There are numerous questions that these concepts naturally lead to.  In particular, a few of them are as follows.  Are there other relationships between smooth and piecewise-linear pseudodiagrams?  Do patterns emerge in WeRe-sets, much like those found in the smooth case \cite{4}?  Are there deeper relationships between the forcing number and piecewise-linear virtual knots?  In \cite{2}, the topology of $n$-sided polygons embedded in $\mathbb{R}^3$, for small values of $n$, is explored.  Understanding any connections between those spaces and forcing number may lead to a better understanding of the space's topology for higher values of $n$.  Lastly, the concept of forcing number may potentially lead to stronger bounds on the edge index \cite{1,8} of PL knots, a fundamental question in PL knot theory.

\end{document}